\documentclass{amsart}
\usepackage{amssymb}
\usepackage{amsfonts}
\usepackage{amssymb}
\usepackage{amsmath, accents}
\usepackage{amsthm}
\usepackage{enumerate}
\usepackage{tabularx}
\usepackage[colorlinks,citecolor=blue,urlcolor=blue, linkcolor=blue]{hyperref}
\usepackage{centernot}
\usepackage{mathtools}
\usepackage{stmaryrd}
\usepackage{amsthm,amssymb}
\usepackage{etoolbox}
\usepackage{tikz}
\usepackage{amssymb}
\usetikzlibrary{matrix}
\usepackage{tikz-cd}
\usepackage{tikz}
\usepackage{marginnote}
\definecolor{mygray}{gray}{0.85}
\usepackage[backgroundcolor=mygray,colorinlistoftodos,prependcaption,textsize=small]{todonotes}
\usepackage{xcolor}

\renewcommand{\leq}{\leqslant}
\renewcommand{\geq}{\geqslant}
\renewcommand{\nleq}{\nleqslant}

\makeatletter
\def\subsection{\@startsection{subsection}{3}%
  \z@{.5\linespacing\@plus.7\linespacing}{.3\linespacing}%
  {\bfseries\centering}}
\makeatother

\makeatletter
\def\subsubsection{\@startsection{subsubsection}{3}%
  \z@{.5\linespacing\@plus.7\linespacing}{.3\linespacing}%
  {\centering}}
\makeatother

\makeatletter
\def\myfnt{\ifx\protect\@typeset@protect\expandafter\footnote\else\expandafter\@gobble\fi}
\makeatother

\newtheorem{theorem}{Theorem}[section]

\newtheorem{definition}[theorem]{Definition}

\newtheorem{observation}[theorem]{Observation}
\newtheorem{fact}[theorem]{Fact}

\newtheorem{remark}[theorem]{Remark}

\newtheorem{notation}[theorem]{Notation}

\newtheorem{cclaim}[theorem]{Claim}

\newtheorem{convention}[theorem]{Convention}

\newtheorem*{nproblem}{Problem}
\newtheorem*{nnotation}{Notation}

\newtheorem*{ntheorem1}{Theorem~1.1}
\newtheorem*{ntheorem2}{Theorem~1.2}
\newtheorem*{ntheorem3}{Theorem~1.3}

\newcounter{claimcounter}
\numberwithin{claimcounter}{theorem}

\usepackage{eso-pic}
\usepackage{blindtext} 

\newcommand{\mrm}[1]{\mathrm{#1}}

\usepackage{pdfpages}
\usepackage{hyperref}

\begin{document}

\dedicatory{This paper is dedicated to Moshe Jarden, in honor of his many contributions to field arithmetic.}

\begin{abstract} We deal with the problem of existence of uncountable co-Hopfian abelian groups and (absolute) Hopfian abelian groups. Firstly, we prove that there are no co-Hopfian reduced abelian groups $G$ of size $< \mathfrak{p}$ with infinite $\mrm{Tor}_p(G)$, and that in particular there are no infinite reduced abelian $p$-groups of size $< \mathfrak{p}$. Secondly, we prove that if $2^{\aleph_0} < \lambda < \lambda^{\aleph_0}$, and $G$ is abelian of size $\lambda$, then $G$ is not co-Hopfian. Finally, we prove that for every cardinal $\lambda$ there is a torsion-free abelian group $G$ of size $\lambda$ which is absolutely Hopfian, i.e., $G$ is Hopfian and $G$ remains Hopfian in every forcing extensions of the universe.
\end{abstract}

\title[On the Existence of Uncountable (co-)Hopfian Abelian Groups]{On the Existence of Uncountable Hopfian and co-Hopfian Abelian Groups}

\thanks{No. 1214 on Shelah's publication list. Research of both authors partially supported by NSF grant no: DMS 1833363. Research of first author partially supported by project PRIN 2017 ``Mathematical Logic: models, sets, computability", prot. 2017NWTM8R. Research of second author partially supported by Israel Science Foundation (ISF) grant no: 1838/19}

\author{Gianluca Paolini}
\address{Department of Mathematics ``Giuseppe Peano'', University of Torino, Via Carlo Alberto 10, 10123, Italy.}

\author{Saharon Shelah}
\address{Einstein Institute of Mathematics,  The Hebrew University of Jerusalem, Israel \and Department of Mathematics,  Rutgers University, U.S.A.}

\date{\today}
\maketitle

%

\section{Introduction}

	A group $G$ is said to be Hopfian (resp. co-Hopfian) if every onto (resp. $1$-to-$1$) endomorphism of $G$ is $1$-to-$1$ (resp. onto), equivalently $G$ is Hopfian if it has no proper quotient isomorphic to itself and co-Hopfian if it has no proper subgroup isomorphic to itself. For example, $\mathbb{Z}$ is Hopfian but not co-Hopfian, while the Pr\"ufer $p$-group $\mathbb{Z}(p^\infty)$ is co-Hopfian but not Hopfian. The notions of  Hopfian and co-Hopfian groups have been studied for a long time, under different names. In the context of abelian group theory they were first considered by Baer in \cite{baer}, where he refers to them as $Q$-groups and $S$-groups. The modern terminology arose from the work of the German mathematician H. Hopf, who showed that the defining property holds of certain two dimensional manifolds. The research on Hopfian and co-Hopfian abelian groups has recently been revived thanks to its recently discovered connections with the study of algebraic entropy and its dual (see \cite{entropy1, entropy2}), as e.g. groups of zero algebraic entropy are necessarily co-Hopfian (for more on the connections between these two topics see \cite{entropy3}). In this paper we will focus exclusively on abelian groups and for us ``group'' will mean ``abelian group''. 
	
	We briefly recall the relevant state of the art in this area and then introduce our motivation and state our theorems. We start by considering the co-Hopfian property. An easy observation shows that a torsion-free abelian group is co-Hopfian if and only if it is divisible of finite rank, hence the problem naturally reduces to the torsion and mixed cases. A major progress in this line of research was made by Beaumont and Pierce in \cite{b&p} where the authors proved several general important results, in particular that if $G$ is co-Hopfian, then $\mrm{Tor}(G)$ is of size at most continuum, and further that $G$ cannot be a $p$-groups of size $\aleph_0$. This naturally left open the problem of existence of co-Hopfian $p$-groups of uncountable size $\leq 2^{\aleph_0}$, which was later solved by Crawley \cite{crawley} who proved that there exist  $p$-groups of size $2^{\aleph_0}$. But the question remained: what about $p$-groups $G$ of size $\aleph_0 < |G| < 2^{\aleph_0}$? Interestingly enough this question remained open until recently, when it was shown by Braun and Str\"ungmann \cite{independence_paper} that this is independent from $\mrm{ZFC}$. Finally, at the best of our knowledge there are no results on the existence of co-Hopfian groups of size $>2^{\aleph_0}$ (and possibly for a ``good reason'', see the discussion after Theorem~\ref{first_main_th}).
	 
	 Moving to Hopfian groups, the situation is quite different, most notably (improving a result of Fuchs \cite{fuchs_indec}, who proved this for $\lambda <$ the first beautiful cardinal) the second author showed in \cite{sh44} that for every infinite cardinal $\lambda$ there is an endorigid torsion-free group of cardinality $\lambda$, i.e., a group $G$ such that for every endomorphism $f$ of $G$ there is $m_f \in \mathbb{Z}$ such that $f(x) = m_fx$ (and such that $f$ is onto iff $m_f \in \{1, -1\}$), evidently such groups are Hopfian and so there are Hopfian groups in every cardinality (recall that finite groups are Hopfian). Hence, the existence of Hopfian groups seems to be settled, but the construction from \cite{sh44} uses stationary sets, so one may wonder about the ``effectiveness'' of the construction from \cite{sh44} or any other known construction of arbitrarily large Hopfian groups. We focus  here on a specific notion of ``effectiveness'' which was suggested for abelian groups by Nadel in \cite{nadel}, i.e., the preservation under any forcing extension of the universe $\mathrm{V}$. We refer to this as the problem of absolute existence (of a group satisfying a certain property). These kind of problems were considered by Fuchs, G\"obel, Shelah and others (see e.g. \cite{678, 4_primes, abs_ind}), probably the most important problem in this area is the problem of existence of absolutely indecomposable groups in every cardinality which remains open to this day (despite several partial answers are known).
	 
	\smallskip

	Relying on the picture sketched above in this paper we consider three major problems on the existence of Hopfian and co-Hopfian groups, namely:
	\begin{nproblem}
	\begin{enumerate}[(1)]
		\item Despite the known necessary restrictions, can we improve (in $\mrm{ZFC}$!) the result from \cite{b&p} that there are no co-Hopfian $p$-groups of size $\aleph_0$ or $> 2^{\aleph_0}$?
	\item Are there co-Hopfian groups in every (resp. arbitrarily large) cardinality?
	\item Are there absolutely Hopfian groups in every cardinality?
\end{enumerate}
\end{nproblem}
	
	We give solutions to the three problems above with the following three theorems.
	
	\begin{nnotation} We denote by $\mrm{AB}$ the class of abelian groups and by $\mrm{TFAB}$ the class of torsion-free abelian groups. Also, given a cardinal $\lambda$, we denote by $\mrm{AB}_\lambda$ the class of $G \in \mrm{AB}$ of cardinality $\lambda$ and by $\mrm{TFAB}_\lambda$ the class of $G \in \mrm{TFAB}$ of cardinality~$\lambda$.
\end{nnotation}

\begin{theorem}\label{main_th1} Suppose that $G \in \mrm{AB}$ is reduced and $\aleph_0 \leq |G| < 2^{\aleph_0}$. If $\mathfrak{p} > |G|$ and there is a prime $p$ such that $\mrm{Tor}_p(G)$ is infinite, then $G$ is not co-Hopfian. In particular there are no infinite reduced co-Hopfian $p$-groups $G$ of size $\aleph_0 \leq |G| < \mathfrak{p}$.
\end{theorem}

	\begin{theorem}\label{main_th2} If $2^{\aleph_0} < \lambda < \lambda^{\aleph_0}$, and $G \in \mrm{AB}_\lambda$, then $G$ is not co-Hopfian.
\end{theorem}

\begin{theorem}\label{first_main_th} For all $\lambda \in \mrm{Card}$ there is $G \in \mrm{TFAB}_\lambda$ which is absolutely Hopfian.
\end{theorem}

	We comment on the theorems above. Theorem \ref{first_main_th} can be considered conclusive in some respect (see also Remark~\ref{remark_4prime}), while Theorems \ref{main_th1} and \ref{main_th2} leave room for further investigations. First of all, Theorem  \ref{main_th1} gives important new information also on the countable case, and in fact in our work in preparation on countable co-Hopfian groups \cite{F2005} we crucially base our investigations upon this result. Also concerning  Theorem \ref{main_th1}, we might ask: is $\mathfrak{p}$ the right cardinal invariant of the continuum? The answer to this question is: yes, but not quite. In a work in preparation \cite{F2057} the second author introduces some new $\mathfrak{p}$-like cardinal invariants of the continuum that are tailored exactly to this purpose. Finally, Theorem~\ref{main_th2} leaves open the question of existence of arbitrarily large co-Hopfian groups, in another work in preparation \cite{F2044} the second author deals with questions surrounding this problem. Finally, a last word on the existence of arbitrarily large co-Hopfian groups: in \cite{F2066} the second author proves that there are no arbitrarily large absolutely co-Hopfian groups, in fact he proves that there are no such groups above the first beautiful cardinal, and so a construction of arbitrarily large co-Hopfian groups has to necessarily use some ``non-effective methods'', such as e.g. Black Boxes \cite{Sh:e}.

	As briefly mentioned above, in a work in preparation \cite{F2005} we deal with classification and anti-classification results for countable co-Hopfian groups, from the point of view of descriptive set theory, extending on results of our recent paper \cite{1205}.

	The structure of the paper is simple, in Section 2 we introduce the necessary notations and preliminaries, in Section 3 we prove Theorem~\ref{main_th1}, in Section 4 we prove Theorem~\ref{main_th2}, and finally in Section 5 we prove Theorem~\ref{first_main_th}.

\section{Notations and preliminaries}

	For readers of various backgrounds, we collect here a number of definitions, notations and (well-known) facts which will be used in the proofs of our theorems. 

	\begin{notation} We denote by $\mathbb{P}$ the set of prime numbers.
\end{notation}

\begin{notation} Let $G$ and $H$ be groups.
\begin{enumerate}[(1)]
	\item $H \leq G$ means that $H$ is a subgroup of $G$.
	\item We let $G^+ = G \setminus \{  0_G \}$, where $0_G = 0$ is the neutral element of $G$.
\end{enumerate}
\end{notation}

	\begin{definition}\label{def_pure} Let $H \leq G \in \mrm{AB}$, we say that $H$ is pure in $G$, denoted by $H \leq_* G$, when if $k \in H$, $n < \omega$ and $G \models ng = k$, then there is $h \in H$ such that $H \models nh = g$.
\end{definition}

	\begin{observation}\label{obs_pure_TFAB} If $H \leq_* G \in \mrm{TFAB}$, $k \in H$ and $0 < n < \omega$, then:
$$G \models ng = k \Rightarrow g \in H.$$
\end{observation}

	\begin{observation}\label{generalG1p_remark} Let $G \in \mathrm{TFAB}$ and let:
	$$G_{(1, p)} = \{ a \in G_1 : a \text{ is divisible by $p^m$, for every $0 < m < \omega$}\},$$
then $G_{(1, p)}$ is a pure subgroup of $G_1$.
\end{observation}

	\begin{proof} This is well-known, see e.g. the discussion in \cite[pg. 386-387]{harrison}.
\end{proof}

	\begin{notation} Given $G \in \mrm{AB}$ and $p \in \mathbb{P}$, we denote by $\mrm{Tor}(G)$ the torsion subgroup of $G$ and by $\mrm{Tor}_p(G)$ the $p$-torsion subgroup of $G$.
\end{notation}

	\begin{notation} Given $G \in \mrm{AB}$, $g \in G$ and $p \in \mathbb{P}$, we write $p^\infty \vert \, g$ to mean that $g$ is $p^n$-divisible (in $G$) for every $0 < n < \omega$.
\end{notation}

	\begin{definition} We say that $G$ has bounded exponent or simply that $G$ is bounded if there is $n < \omega$ such that $nG = \{0\}$.
\end{definition}

	\begin{notation} Given $G, H \in \mrm{AB}$, we denote by $\mrm{Hom}(G, H)$, the set of homomorphisms between $G$ and $H$. We denote by $\mrm{End}(G)$ the set $\mrm{Hom}(G, G)$.
\end{notation}

	Items \ref{kapla1}-\ref{suff_cond} below, which will be used in Section~\ref{sec_cohop_char}, are well-known to group theorists, we state them here (with references) for completeness of exposition. 

	\begin{fact}[{\cite[pg.~18, Theorem~7]{kaplansky}}]\label{kapla1} Let $G \in \mrm{AB}$ and $H$ a pure subgroup of $G$ of bounded exponent.  Then $H$ is a direct summand of $G$.
\end{fact}

	\begin{fact}[{\cite[pg.~18, Theorem~8]{kaplansky}}]\label{kapla2} Let $G \in \mrm{AB}$ and $T \leq_* \mrm{Tor}(G)$. If $T$ is the direct sum of a divisible group and a group of bounded exponent, then $T$ is a direct summand of $G$.
\end{fact}


	\begin{fact}[\cite{b&p}]\label{co-hop_pgroup} Let $G \in \mrm{AB}$ be a countable $p$-group. Then $G$ is co-Hopfian if and only if $G$ is finite.
\end{fact}

	\begin{fact}[{\cite[Theorem~17.2]{fuch_vol1}}]\label{bounded_exponent} If $A \in \mrm{AB}$ is a $p$-group of bounded exponent, then $A$ is a direct sum of (finitely many, up to isomorphism) finite cyclic groups.
\end{fact}

	\begin{fact}\label{fact6A} Let $G \in \mrm{AB}$ and $p \in \mathbb{P}$.
	\begin{enumerate}[(1)]
	\item If $G$ is an unbounded $p$-group, then $G$ has a pure cyclic subgroup of arbitrarily large finite size.
	\item $\mrm{Tor}_p(G) \leq_* \mrm{Tor}(G) \leq_* G$ (for $\leq_*$ cf. Definition~\ref{def_pure}).
\end{enumerate}
\end{fact}

	\begin{proof} (1) follows by \ref{kapla1} and (2) is well-known.
\end{proof}


	\begin{cclaim}\label{c7} Let $G \in \mrm{AB}$. Then:
	\begin{enumerate}[(1)]
	\item\label{c7(1)} If $G = G_1 \oplus G_2$ and $G_1$ is not co-Hopfian (resp. Hopfian), then $G$ is not co-Hopfian (resp. Hopfian); 
	\item If $G \in \mrm{TFAB}$, then $G$ is co-Hopfian iff $G$ is divisible of finite rank;
	\item If $G = G_1 \oplus G_2$ and $G_1 \neq \{0\} \neq G_2$, then $G$ has a non-trivial automorphism.
	\end{enumerate}
\end{cclaim}

	\begin{proof} Each item is either easy or well-known, see e.g. \cite{b&p}.
	\end{proof}
	
	\begin{fact}\label{extending_gtoh} Let $K \in \mrm{AB}$ be a bounded torsion group and let $G \leq_* H \in \mrm{AB}$. If $g \in \mrm{Hom}(G, K)$, then there is $h \in \mrm{Hom}(H, K)$ extending $g$.
\end{fact}

	\begin{proof} This is because, by Fact~\ref{kapla1}, $K$ is algebraically compact (cf. \cite[Section~38]{fuch_vol1}) and such groups are exactly the pure-injective groups in $\mrm{AB}$ (see \cite[Theorem~38.1]{fuch_vol1}).
\end{proof}
	
	\begin{observation}\label{suff_cond} Let $G \in \mrm{AB}$. Then $G$ is non-co-Hopfian if and only if:
	\begin{enumerate}[$(\star)$]
	\item there are $f$ and $z  \in G$ such that:
	\begin{enumerate}[(a)]
	\item $f \in \mrm{End}(G)$;
	\item $f(x) \neq x$ for every $x \in G \setminus \{ 0 \}$;
	\item for every $x \in G$, $z \neq x - f(x)$.
	\end{enumerate}
	\end{enumerate}
\end{observation}

	\begin{proof} For the direction right-to-left, notice that letting $g = id_G - f \in \mrm{End}(G)$ we have that by (b) $g$ is $1$-to-$1$ and by (c) $g$ is not onto. The other direction is also easy as if $g$ is a witness for non-co-Hopfianity of $G$, then $\mrm{id}_G - g$ satisfies $(\star)$.
\end{proof}

	The following notation will be relevant in Section~\ref{last_sec}.

\begin{notation}\label{inf_logic_notation} By $\tau_{\mrm{AB}}$ we denote the vocabulary of abelian groups $\{+, -, 0\}$. Given $\lambda, \kappa \in \mrm{Card}$ we denote by $\mathfrak{L}_{\lambda, \kappa}(\tau_{\mrm{AB}})$ the corresponding infinitary $\tau_{\mrm{AB}}$-formulas (see e.g. \cite{dickman}). Sometimes we simply write $\varphi \in \mathfrak{L}_{\kappa, \lambda}$ instead of $\varphi \in \mathfrak{L}_{\kappa, \lambda}(\tau_{\mrm{AB}})$.
\end{notation}

	We now introduce the cardinal invariant $\mathfrak{p}$ (which occurs in Theorem~\ref{main_th1}).
	
	\begin{definition} The cardinal invariant of the continuum $\mathfrak{p}$ is the minimum size of a family $\mathcal{F}$ of infinite subsets of $\omega$ such that:
	\begin{enumerate}[(i)]
	\item every non-empty finite subfamily of $\mathcal{F}$ has infinite intersection;
	\item there is no infinite $A \subseteq \omega$ s.t., for every $B \in \mathcal{F}$, $\{x \in A : x \notin B \}$ is finite.
	\end{enumerate}
\end{definition}

	\section{Co-Hopfian abelian groups of size $\aleph_0 < \lambda < 2^{\aleph_0}$}\label{sec_cohop_char}

	As mentioned in the introduction, in this section we aim at proving Theorem~\ref{main_th1}, to this extent we first prove Claim~ \ref{the_crucial_claim} which deals with the countable case and then detail on how to modify the proof in order to get Claim~\ref{the_crucial_claim_post} which gives Theorem~\ref{main_th1}.

	\begin{remark} By \ref{c7}, the assumption ``$G$ is reduced'' is without loss of generality. This applies e.g. to Claim~\ref{the_crucial_claim} below.
\end{remark}
	
	\begin{cclaim}\label{the_crucial_claim} Let $G \in \mrm{AB}$ be countable and reduced. Let also $p \in \mathbb{P}$, and suppose that $\mrm{Tor}_p(G)$ is infinite. Then:
	\begin{enumerate}[(1)]
	\item $G$ is not co-Hopfian;
	\item If in addition $\mrm{Tor}_p(G)$ is not bounded, then we can find $\bar{K}$ and $K$ such that:
	\begin{enumerate}
	\item $\bar{K} = (K_n : n < \omega)$ and $K = \bigoplus_{n < \omega} K_n \leq_* G$;
	\item $K_n \leq G$ is a non-trivial finite $p$-group;
	\item there is $f \in \mrm{End}(G)$ such that $\mrm{ran}(f) \subseteq K$ and for every $n < \omega$ we have that $\{0\} \neq f(K_n) \subseteq K_n$;
	\item $f$ is as in \ref{suff_cond}.
	\end{enumerate}
	\item If in addition to (2) $\mrm{Tor}_p(G)$ has height $\geq \omega$, then in (2)(b) we have that for some increasing $k(n)$,  $p^n(p^{k(n)}K_n) \neq \{ 0 \}$ and $x \in K_n \Rightarrow f(x) = p^{k(n)}(x)$.
	\end{enumerate}
\end{cclaim}

	\begin{proof} If $\mrm{Tor}_p(G)$ is a bounded infinite group, then by Fact~\ref{bounded_exponent} we have part (1). So assume that $\mrm{Tor}_p(G)$ is infinite and not bounded, we prove Items (2) and (3) simultaneously, as Item (2) implies (1) by Observation~\ref{suff_cond}, recalling (1)(d), this suffices. As $\mrm{Tor}_p(G)$ is infinite and not bounded, we can choose $(L_n, H_{n+1}, y_n)$ s.t.:
	\begin{enumerate}[$(*_1)$]
	\item
	\begin{enumerate}[(a)]
	\item $H_0 = G$;
	\item $H_{n} = H_{n+1} \oplus L_n$;
	\item $L_n = \mathbb{Z} y_n$ and $y_n$ has order $p^{\ell(n)}$, so $L_n, H_{n+1}, H_n \leq_* G$;
	\item without loss of generality $\ell(n) \leq \ell(n+1)$;
	\item moreover $\ell(n) < \ell(n+1)$.
	\end{enumerate}
	\end{enumerate}
	[Why we can do this? By induction on $n < \omega$, using Facts~\ref{kapla2} and \ref{fact6A}.]
	\begin{enumerate}[$(*_2)$]
	\item We can find $f_0 \in \mrm{End}(G)$ such that:
	\begin{enumerate}
	\item $f_0$ maps $H_2$ into itself;
	\item $f_0$ maps $L_0 \oplus L_1$ into itself;
	\item for some $z \in L_0 \oplus L_1$, $z \notin \{ x - f_0(x) : x \in L_0 \oplus L_1 \}$;
	\item $x \in L_0 \oplus L_1$ implies $x \neq f_0(x)$.
\end{enumerate}	
\end{enumerate}
[Why? First, $G = H_2 \oplus L_1 \oplus L_0$. Now, let $f_0 \restriction H_2$ be $0$ and $f_0 \restriction L_0 \oplus L_1$ be defined by $f_0(m_0 y_0 + m_1y_1) = p^{\ell(1)-\ell(0)} m_0 y_1$. Then $f_0$ is as wanted letting $z = y_0$.]
\begin{enumerate}[$(*_3)$]
	\item if $A \subseteq G$ is finite and $n_0 < \omega$, then we can find $n_2 > n_0$ and $h$ such that:
	\begin{enumerate}[(a)]
	\item $h$ is an hom. from $G$ onto $\mathbb{Z}(p^{\ell(n_2) - \ell(n_0)} y_{n_2})$ (cyclic grp. of order $p^{\ell(n_0)}$);
	\item $h(a) = 0$, for $a \in A$;
	\item $h(y_{n_2}) = p^{\ell(n_2) - \ell(n_0)} y_{n_2}$
	\item $n_2 - n_0 \leq (p^{\ell(n_0)})^{|A|}$;
	\item $h(y_\ell) = 0$, for $\ell < n_0$.
\end{enumerate}
\end{enumerate}
[Why? By $(*_1)$, for each $n < \omega$ we can find a projection $h_n$ of $G$ onto $\mathbb{Z} y_n$, mapping $y_0, ..., y_{n-1}$ to zero and $y_n$ to $y_n$. So, for every $n_0 \leq n < \omega$, $h'_{(n, n_0)} := p^{\ell(n) - \ell(n_0)}h_n$ is an homorphism from $G$ onto $\mathbb{Z}(p^{\ell(n) - \ell(n_0)} y_{n})$, which has order $p^{\ell(n_0)}$.
 Moreover, fixed $n < \omega$, for every $a \in A$ there is $m_n(a) \in \{0, ..., p^{\ell(n_0)}- 1\}$ such that $h'_{(n, n_0)}(a) = m_n(a) p^{\ell(n) - \ell(n_0)} y_n$ (recall that $\mathbb{Z}(p^{\ell(n) - \ell(n_0)} y_{n})$ has order $p^{\ell(0)}$). Thus, by the pigeon-hole principle there are $n_1,  n_2 < \omega$ such that:
\begin{enumerate}[$(\cdot_1)$]
\item $n_0 \leq n_1 < n_2 \leq n_0 + (p^{\ell(n_0)})^{|A|}$;
\item if $a \in A$, then $m_{n_1}(a) = m_{n_2}(a)$.
\end{enumerate}
Now, let $h \in \mrm{End}(G)$ be defined as follows:
$$h(x) = h'_{(n_2, n_0)}(x) - f'_{(n_2, n_1)}(h'_{(n_1, n_0)}(x)),$$
where $f'_{(n_2, n_1)} (my_{n_1}) = m p^{\ell(n_2) -\ell(n_1)} y_{n_2}$, for $m \in \mathbb{Z}$. Then $h$ is as wanted in $(*_3)$.]

\smallskip
\noindent Now we can finish the proof. Let $(a_i : i < \omega)$ list the elements of $G$. We choose $(f_i, K_i, k_i, m_i, n_i)$ by induction on $i < \omega$ as follows:
	\begin{enumerate}[$(*_4)$]
	\item 
	\begin{enumerate}[(a)]
	\item For $i = 0$, let $f_0$ be as in $(*_2)$ and $K_0 = L_0 \oplus L_1$;
	\item for $i > 0$, $K_i = L_{n(i)}$, $n(i) \geq 2$ and $n_i$ is strictly increasing with $i$;
	\item $f_i \in \mrm{End}(G)$ has range in $K_i$ and:
	$$f_i(y_{n(i)}) = p^{k(i)} y_i, \text{ with } i \leq k(i) = \ell(n(i)) - m(i) < \ell(n(i)),$$
	so that $f_i(y_{n(i)})$ has order $p^{m(i)}$, where $m(i)$ is as in $(*_{3.1})$;
	\item $f_i$ maps $H_2$ into itself and $L_0 \oplus L_1$ to zero;
	\item $f_i$ maps $a_0, ..., a_i$ to $0$.
	\end{enumerate}
	\end{enumerate}
[Why can we carry the induction? For $i = 0$ use $(*_2)$, for $i = j+1$ use $(*)_3$.]
\newline Now, let $K = \bigoplus_{i < \omega} K_i$ define $f$ as follow: for $x \in G$, we let $f(x) = \sum \{f_i(x) : i < \omega\}$. This infinite sum is well-defined because by clause $(*_4)(e)$, $f_i(x) = 0$, for every large enough $i$. It is easy to see that $(f, K)$ are as wanted for clauses (2)(a)-(c) and (3), finally we show that clause (2)(d) holds, i.e., that $f$ satisfies the hypotheses of \ref{suff_cond}. First, \ref{suff_cond}(a) is obvious. Concerning \ref{suff_cond}(b), as $\mrm{ran}(f) \leq K$, clearly if $x \notin K$ we have that $f(x) \neq x$. On the other hand if $x \in K$ use the equation in $(*_4)(c)$, when $x \notin K_0$ and $(*_2)(d)$ when $x \in K_0$. Finally, concerning \ref{suff_cond}(c), let $x \in G$, we want to show that $z \neq x - f(x)$, where $z$ is as in $(*_2)(c)$. Recall that $G = H_2 \oplus L_1 \oplus L_0 = H_2 \oplus K_0 = K_0 \oplus H_2$, so $x = x_1 + x_2$, with $x_1 \in K_0$ and $x_2 \in H_2$. Thus, $f(x) = (x_1 - f(x_1)) + (x_2 - f(x_2))$. If $x_2 - f(x_2) \neq 0$, then clearly $z \neq x - f(x)$, as $z \in K_0$. On the other hand, if $x_2 - f(x_2) = 0$, then $f(x) = x_1 - f(x_1)$ and by $(*_2)(c)$ we are done. This concludes the proof, as (1) is clear and (3) is also satisfied as we can let $m(i) = i$ in $(*_4)$.
\end{proof}

	\begin{cclaim}\label{the_crucial_claim_post} In the context of Claim~\ref{the_crucial_claim}.
	\begin{enumerate}[(A)]
	\item We can omit ``reduced'' if we strengthen ``$\, \mrm{Tor}_p(G)$ is infinite" to ``$\, \mrm{Tor}_p(G)$ is of infinite rank and $\mrm{div}(G) \cap \mrm{Tor}_p(G)$ is of finite rank''.
	\item We can omit ``countable'' if $|G| < 2^{\aleph_0}$ and $\mrm{MA}$ holds or at least $\mathfrak{p} > |G|$.
	\item We can apply both (A) and (B) simultaneously.
	\end{enumerate}
\end{cclaim}

	\begin{proof} Concerning $(A)$, let $G = G_1 \oplus G_2$, with $G_1$ divisible and $G_2$ reduced. As $\mrm{Tor}_p(G) = \mrm{Tor}_p(G_1) \oplus \mrm{Tor}_p(G_2)$ and $\mrm{Tor}_p(G_1)$ is of finite rank, necessarily $\mrm{Tor}_p(G_2)$ is of infinite rank, hence it is in particular infinite and so we can apply Claim~\ref{the_crucial_claim} to $G_2$ and thus conclude by \ref{c7}(\ref{c7(1)}) that $G$ is not co-Hopfian recalling that $G = G_1 \oplus G_2$. 

\smallskip
\noindent	
	Concerning clause (B), it suffices to run the proof of \ref{the_crucial_claim} up to $(*_4)$. First of all, recalling that $\mrm{MA} \wedge 2^{\aleph_0} \Rightarrow \mathfrak{p} > |G|$, we can assume that $\mathfrak{p} > |G|$. Now, let $(y_n : n < \omega)$, $(\ell(n) : n < \omega)$, $(h'_{(m, n)} : n < m < \omega)$ and $(f'_{(m, n)} : n < m < \omega)$ be as in $(*_1)$ and the proof of $(*_3)$ from the proof of \ref{the_crucial_claim}. Now, for every finite $A \subseteq G$ and $n < \omega$, let $X_{(A, n)}$ be the following set:
	$$\{(n_2, n_1, n_0) : n_2 > n_1 \geq n_0 \geq n \text{ and } (h'_{(n_2, n_0)} - f'_{(n_2, n_1)}h'_{(n_1, n_0)})(A) = \{0\}\}.$$
Now, we have:
\begin{enumerate}[$(+_1)$]
	\item
	\begin{enumerate}[(a)]
	\item for $(A, n)$ as above $X_{(A, n)}$ is infinite;
	[Why? By the proof of $(*_3)$ in \ref{the_crucial_claim}.]
	\item if $n \leq m < \omega$ and $A \subseteq B \subseteq_\omega G$, then $X_{(B, m)} \subseteq X_{(A, n)}$.
	\end{enumerate}
\end{enumerate}
As the set $\{(A, n) : A \subseteq_\omega G, n < \omega\}$ has cardinality $|G|$ and by $\mathfrak{p} > |G|$ , recalling that $|G|$ is finite, by the definition of $\mathfrak{p}$ and $(+_1)$ we have:
\begin{enumerate}[$(+_2)$]
	\item there is an infinite $X_* \subseteq \{(n_2, n_1, n_0) : n_0 \leq n_1 < n_2 < \omega\}$ such that for every $(A, n)$ as above we have $X_* \subseteq X_{(A, n)}$ modulo finitely many elements.
\end{enumerate}
Now, by induction on $i < \omega$, choose $(n_{(i, 2)}, n_{(i, 1)}, n_{(i, 1)}) \in X_*$ such that $j < i$ implies $n_{(j, 2)} < n_{(i, 0)}$. Finally, let $f_0 \in \mrm{End}(G)$ be as in $(*_4)(a)$ of the proof of Claim~\ref{the_crucial_claim} and, for $0 < i < \omega$, let $f_i \in \mrm{End}(G)$ be as follows:
$$h'_{(n_{(i, 2)}, n_{(i, 0)})} - f'_{(n_{(i, 2)}, n_{(i, 1)})}h'_{(n_{(i, 1)}, n_{(i, 0)})}$$
Now, let $K = \bigoplus_{i < \omega} K_i$, i.e., as in the proof of \ref{the_crucial_claim}, and let $f$ be such that for $x \in G$ we have $f(x) = \sum \{f_i(x) : i < \omega\}$. Notice that $f$ is well-defined (and so clearly $f \in \mrm{End(G)})$), as for every $x \in G$ we have that $\{i < \omega : f_i(x) \neq 0 \}$ is finite, given that $X_* \subseteq X_{(\{a\}, 0)}$ modulo finite, by construction. It is now easy to see that $(f, K)$ are as wanted, arguing as in the proof of this is as in the proof \ref{the_crucial_claim}.
This concludes the proof of (B), finally clause (C) is by combining the proofs of (A) and (B).
\end{proof}

	We are now in the position to prove our first main theorem.
	
	\begin{ntheorem1} Suppose that $G \in \mrm{AB}$ is reduced and $\aleph_0 \leq |G| < 2^{\aleph_0}$. If $\mathfrak{p} > |G|$ and there is a prime $p$ such that $\mrm{Tor}_p(G)$ is infinite, then $G$ is not co-Hopfian. In particular there are no infinite reduced co-Hopfian $p$-groups $G$ of size $\aleph_0 \leq |G| < \mathfrak{p}$.
\end{ntheorem1}
	
	\begin{proof} Immediate by Claims~ \ref{the_crucial_claim} and \ref{the_crucial_claim_post}.
\end{proof}

	\begin{remark}\label{right_cardinal_invariant} In the context of Claim~\ref{the_crucial_claim_post}(B) we ask ourselves: is $\mathfrak{p}$ the right cardinal invariant? The answer is: yes, but not quite. On this see \cite{F2057}.
\end{remark}

	The following claim is essentially known, in particular (1), see e.g. \cite[pg. 213]{b&p} on this, but we mention it as it follows from the proofs of the claims above.

\begin{cclaim} Let $G \in \mrm{AB}$ be reduced.
	\begin{enumerate}[(1)]
	\item If $\mrm{Tor}_p(G)$ is of cardinality $> 2^{\aleph_0}$, then $G$ is not co-Hopfian;
	\item If $|\mrm{Tor}_p(G)| \geq \lambda$, $\mrm{cof}(\lambda) = \aleph_0$ and $\alpha < \lambda \Rightarrow |\alpha|^{\aleph_0} < \lambda$, then $|\mrm{End}(G)| \geq 2^{\lambda}$.
\end{enumerate}
\end{cclaim}


%


	This claim will be relevant in what follows and it is of independent interest.

	\begin{cclaim}\label{quotient_withT_divisible} Let $G \in \mathrm{AB}$ and $p \in \mathbb{P}$. If $\mrm{Tor}_p(G)$ is bounded and $G/\mrm{Tor}_p(G)$ is not $p$-divisible, then $G$ is not co-Hopfian.
\end{cclaim}

	\begin{proof} Let $K = \mrm{Tor}_p(G)$, then, recalling that by assumption $K$ is bounded, by Fact~\ref{kapla1}, $K$ is a direct summand of $G$, say $G = H \oplus K$. Now, $H \in \mrm{AB}$ and $\mrm{Tor}_p(H) = \{0\}$, hence $x \mapsto px$ is a $1$-to-$1$ endomorphism of $H$ which is not onto (as otherwise $G/\mrm{Tor}_p(G)$ would be $p$-divisible). By Claim~\ref{c7}(1) we are done. 
\end{proof}

\section{Non-existence of co-Hopfian abelian groups}

	As mentioned in the introduction, in this section we aim at proving Theorem~\ref{main_th2}, to this extent we prove two theorems: \ref{no_cohopfian_1} and \ref{no_cohopfian_2}, from which Theorem~\ref{main_th1} follows. Theorem~\ref{no_cohopfian_1} has stronger assumptions and a simpler proof, while Theorem~\ref{no_cohopfian_2} has weaker assumptions but a more complicated proof, but it is needed for Theorem~\ref{main_th2}.


	\begin{theorem}\label{no_cohopfian_1} Suppose that $\lambda = \sum_{n < \omega} \lambda_n > 2^{\aleph_0}$, and, for every $n < \omega$, $\lambda_n = \lambda^{\aleph_0}_n < \lambda_{n+1}$. If $G \in \mrm{AB}_\lambda$, then $G$ is not co-Hopfian.
\end{theorem}

	\begin{proof} The proof splits into cases.
\newline \underline{Case 1}. $\mrm{Tor}_p(G) = \{ 0 \}$ and $pG \neq G$. 
\newline In this case $x \mapsto px$ is a $1$-to-$1$ endomorphism of $G$ which is not onto.
\newline \underline{Case 2}. $|\mrm{Tor}(G)| > 2^{\aleph_0}$.
\newline In this case $G$ is not co-Hopfian, see e.g. \cite{b&p}.
\newline \underline{Case 3}. $G$ has an infinite rank divisible subgroup which is torsion-free or a $p$-group.
\newline This case is easy.
\newline \underline{Case 4}. For some $p \in \mathbb{P}$, $\mrm{Tor}_p(G)$ is finite and $G/\mrm{Tor}_p(G)$ is not $p$-divisible.
\newline Also in this case $G$ is not co-Hopfian, cf. Claim~\ref{quotient_withT_divisible}. 
\newline \underline{Case 5}. For some $p \in \mathbb{P}$, $\mrm{Tor}_p(G)$ is infinite and bounded.
\newline Also in this case $G$ is not co-Hopfian, cf. Claim~\ref{the_crucial_claim}. 

\smallskip
\noindent Hence, recalling \ref{c7}, w.l.o.g. for the rest of the proof we can assume:
\begin{enumerate}[$(+)$]
\item $G$ is reduced and $G$ does not fall under Cases 1, 2, 3, 4, 5.
\end{enumerate}
So we have:
\begin{enumerate}[$(*_0)$]
	\item For each $p \in \mathbb{P}$ we have (a) or (b), where:
	\begin{enumerate}[$(a)_p$]
	\item $\mrm{Tor}_p(G)$ is infinite of cardinality $\leq 2^{\aleph_0}$;
	\item $\mrm{Tor}_p(G)$ is finite and $G/\mrm{Tor}_p(G)$ is $p$-divisible.
	\end{enumerate}
\end{enumerate}
\begin{enumerate}[$(*_1)$]
	\item 
	\begin{enumerate}[(a)]
	\item Let $\mathbb{A} = \{p \in \mathbb{P}: \mrm{Tor}_p(G) \text{ is infinite}\}$;
	\item For every $p \in \mathbb{A}$ there is $K_p = \bigoplus \{K_{(p, n)} : n < \omega \} \leq_* \mrm{Tor}_p(G)$ such that for every $n < \omega$, $K_{(p, n)} \cong \mathbb{Z}_{p^{k(p, n)}} z_{(p, n)}$, with $1 \leq k(p, n) < \omega$ and $k(p, n)$ increasing with $n$, as $p \in \mathbb{A}$ and not Case~5 $\Rightarrow \mrm{Tor}_p(G)$ is not bounded.
	\end{enumerate}
	[Why we can get $K_p = \bigoplus \{K_{(p, n)} : n < \omega \} \leq_* \mrm{Tor}_p(G)$? By Fact~\ref{fact6A}(2).]
\end{enumerate}
\noindent In $(*_7)$ below we will prove that $\mathbb{A} \neq \emptyset$. Now we move to:
\begin{enumerate}[$(*_2)$]
	\item Choose $(G_n : n < \omega)$ such that:
	\begin{enumerate}[(a)]
	\item $\bigcup_{n < \omega} G_n = G$;
	\item for every $n < \omega$, $G_n \leq G_{n+1} \leq G$ and $|G_n| \leq \lambda_n$;
	\item $G_n \preccurlyeq_{\mathfrak{L}_{\aleph_1, \aleph_1}} G$;
	\item for every $n < \omega$, if $(a_\ell : \ell < \omega) \in G_n^{\omega}$, $(x_\ell : \ell < \omega) \in G^\omega$, $(k_\ell : \ell < \omega) \in \mathbb{Z}$ and, for every $\ell < \omega$, $x_\ell = k_\ell x_{\ell+1} + a_\ell$, then for some $(y_\ell : \ell < \omega) \in G^\omega_n$ we have that, for every $\ell < \omega$, $y_\ell = k_\ell y_{\ell+1} + a_\ell$;
	\item $G_n \leq_* G$;
	\item $\mrm{Tor}(G) \leq G_0$;
	\end{enumerate}
	\end{enumerate}
[Why $(*_2)$ holds? We can fulfill (a)-(b) because we assume that $\lambda = \sum_{n < \omega} \lambda_n$, and we can fulfill (c) because we assume $\lambda_n = \lambda^{\aleph_0}_n$ (see e.g. \cite[Corollary~3.1.2]{dickman}). Items (d) and (e) follow from (c). Finally, we can fulfill (f) easily recalling that $|\mrm{Tor}_p(G)| \leq 2^{\aleph_0}$ and that by assumption $\lambda_0 = \lambda_0^{\aleph_0}$, which implies that $\lambda_0 \geq 2^{\aleph_0}$.]                   
\begin{enumerate}[$(*_3)$]
	\item Choose $(H_n : n < \omega)$ such that:
	\begin{enumerate}[(a)]
	\item $G_n \leq H_n \leq G_{n+1}$;
	\item $H_n$ is a pure subgroup of $G$;
	\item $H_n/G_n$ is torsion-free of rank $1$.
	\end{enumerate}
\end{enumerate}
[Why possible? Let $a_n \in G_{n+1} \setminus G_n$ and let $H_n$ be the pure closure of $G_n + \mathbb{Z}a_n$, then recalling $(*_2)(a)$ and $(*_2)(f)$ we are done.]
\newline From here till $(*_8)$ excluded, fix $n < \omega$.
\begin{enumerate}[$(*_4)$]
	\item Let $h_n \in \mrm{Hom}(H_n, \mathbb{Q})$ be such that $h_n \neq 0$ and $\mrm{ker}(h_n) = G_n$.
\newline	[Why possible? By $(*_3)$.]
\end{enumerate}
\begin{enumerate}[$(*_5)$]
	\item There is an homomorphism $g_n: \mrm{ran}(h_n) \rightarrow H_n$ be such that $h_n \circ g_n = id_{\mrm{ran}(h_n)}$.
\end{enumerate}
We prove $(*_5)$. Let $q_{(n, \ell)} \in \mrm{ran}(h_n)$ be such that:
\begin{enumerate}[$(\cdot_1)$]
	\item $\mathbb{Z}q_{(n, \ell)} \subseteq \mathbb{Z} q_{(n, \ell+1)} \subseteq \mrm{ran}(h_n)$;
	\item $\bigcup_{\ell < \omega} \mathbb{Z} q_{(n, \ell)} = \mrm{ran}(h_n)$.
	\end{enumerate}
Let $q_{(n, \ell)} = k_{(n, \ell)} q_{(n, \ell+1)}$, with $1 \leq k_{(n, \ell)} < \omega$. Let $x_{(n, \ell)}$ be such that $h_n(x_{(n, \ell)}) = q_{(n, \ell)}$. Thus, for each $\ell < \omega$ we have:
	\begin{enumerate}[$(*_{5.1})$]
	\item $h_n(k_{(n, \ell)} x_{(n, \ell+1)} - x_{(n, \ell)}) = k_{(n, \ell)} q_{(n, \ell+1)} - q_{(n, \ell)} = 0$,
	\end{enumerate}
which means that $a_{(n, \ell)} := k_{(n, \ell)} x_{(n, \ell+1)} - x_{(n, \ell)} \in G_n$. By $(*_2)(c)$, there are $y_{(n, \ell)} \in G_n$ such that for $\ell < \omega$ we have $a_{(n, \ell)} = k_{(n, \ell)} y_{(n, \ell+1)} - y_{(n, \ell)}$. Now, define $g_n : \mrm{ran}(h_n) \rightarrow H_n$ as follows, for $\ell < \omega$ and $m \in \mathbb{Z}$, we let:
$$g_n(m q_{(n, \ell)}) = m(x_{(n, \ell)} - y_{(n, \ell)}),$$
clearly $g_n$ is well-defined and it is $1$-to-$1$ homomorphism from $\mrm{ran}(h_n)$ into $H_n$.
\begin{enumerate}[$(*_6)$]
	\item
	\begin{enumerate}[(a)]
	\item $H_n = G_n \oplus L_n$, where $L_n = \mrm{ran}(g_n)$;
	\item $L_n$ is torsion-free of rank $1$.
	\end{enumerate}
\end{enumerate}
[Why? As $g_n$ is $1$-to-$1$ and $G_n \cap L_n = \{0\}$.]
\begin{enumerate}[$(*_7)$]
	\item
	\begin{enumerate}[(a)]
	\item $L_n$ is not divisible;
	\item there is a prime $p_n$ such that $p_n L_n \neq L_n$;
	\item we can choose $y_n \in L_n$ not divisible by $p_n$ (in $L_n$ and even in $G$);
	\item $p_n \in \mathbb{A}$.
	\end{enumerate}
\end{enumerate}
[Why? Item (a) is because of (+)  (see the beginning of the proof), which implies in particular that $G$ is not reduced, recalling that $\{ 0 \} \neq L_n \leq G$. Item (b) is by (a). Item (c) is because by (b) we can choose $y_n \in L_n$ as required (as $L_n \leq_* H_n \leq_* G$, by $(*_6)(a)$ and $(*_3)(b)$, respectively). Lastly, (d) is because by (b) we have  $G/\mrm{Tor}_{p_n}(G)$ is not $p_n$-divisible, recalling the definition of $\mathbb{A}$. This proves $(*_7)$.]
	\begin{enumerate}[$(*_8)$]
	\item For $n < \omega$, recalling $(*_1)$, let:
	\begin{enumerate}[(a)]
	\item if $n = 0$, then $K_n = K_{(p_n, 0)} \oplus K_{(p_n, 1)}$;
	\item if $n > 0$, then $K_{n} = K_{(p_n, n+1)}$;
	\item $K = \bigoplus \{ K_n : n < \omega \}$.
	\end{enumerate}
	\end{enumerate}
	\begin{enumerate}[$(*_{8.1})$]
	\item $K_n \leq_* \mrm{Tor}_{p_n}(G) \leq_* G$.
	\end{enumerate}
	[Why? The first is by $(*_1)(b)$ and the second is by Fact~\ref{fact6A}.]
	\begin{enumerate}[$(*_9)$]
		\item Let $n < \omega$, then:
		\begin{enumerate}
		\item 
		\begin{enumerate}[(i)]
		\item for $n = 0$, we let:
		\begin{enumerate}[$(\cdot_1)$]
		\item $h^0_0 \in \mrm{End}(K_0)$ be such that:
		$$h^0_0(z_{(p_0, 0)}) = z_{(p_0, 0)} + p^{k(p_0, 1) - k(p_0, 0)} z_{(p_0, 1)}$$
		$$h^0_0(z_{(p_0, 1)}) = z_{(p_0, 1)};$$
		\item $z = z_{(p_0, 0)}$, so $x \in K_0 \Rightarrow x - h^0_0(x) \neq z$, as in the proof of~\ref{the_crucial_claim};
		\item let $f^0_0 \in \mrm{Hom}(G_0, K_0)$ extend $h^0_0$;
		\end{enumerate}
		\item for $n > 0$, we let $f^0_n \in \mrm{Hom}(G_n)$ be zero;
		\end{enumerate}
		\item there is onto $f^1_n \in \mrm{Hom}(L_n, K_n)$ mapping $y_n$ to $z_{(p_n, n)}$;
		\item there is $f^2_n \in \mrm{Hom}(H_n, K_n)$ extending $f^1_n$ and $f^0_n$;
		\item there is $f^3_n = f_n \in \mrm{Hom}(G, K_n)$ extending $f^2_n$.
		\end{enumerate}
	\end{enumerate}
[Why? (a)$(\cdot_1)$-$(\cdot_2)$ is clear and (a)$(\cdot_3)$ is by \ref{extending_gtoh} recalling $(*_{8.1})$. Concerning (b), for every $\ell < \omega$, there are $y_{(n, \ell)} \in L$ such that $(\mathbb{Z}y_{(n, \ell)} : \ell < \omega)$ is increasing with union $L$ and $y_{(n, 0)} = y_n$, so let $y_{(n, \ell)} = m_{(n, \ell)} y_{(n, \ell+1)}$ with $1 \leq m_{(n, \ell)} < \omega$ and $(m_{(n, \ell)}, p_n) = 1$, by the choice of $y_n$. Now, let $b_{(n, \ell)} \in K_n$ such that:
$$b_{(n, 0)} = z_{(p_n, n)} \text{ and }\ell = k+1 \Rightarrow b_{(n, k)} = m_{(n, k)} b_{(n, \ell)},$$
where for $\ell = k+1$ we use that $K_n$ is divisible by $m_{(n, \ell)}$ as $(m_{(n, \ell)}, p_n) = 1$. This proves (b).
Clause (c) is because $H_n = G_n \oplus L_n$. Finally, clause (d) is by Fact~\ref{extending_gtoh}.]
Now we can continue as in the proof of Claim~\ref{the_crucial_claim}, specifically, we define $f$ as follows:
$$f(x) = \sum \{f_n(x) : n < \omega \}.$$
This infinite sum is well-defined because by $(*_9)(c)$, $n >0$, $x \in G_n$ implies $f_n(x) = f^2(x) = 0$ and by $(*_2)$, $x \in G$ implies for almost all $n < \omega$, $x \in G_n$. Now we claim that $f$ as in as \ref{suff_cond}. Preliminarily, notice that for every $n < \omega$ we have that:
	\begin{enumerate}[$(*_{10})$]
	\item $f$ maps $G$ into $K$ and $K \leq G_0$.
	\end{enumerate}
Now, returning to showing that $f$ as in as \ref{suff_cond}, Item \ref{suff_cond}(a) is obvious, concerning \ref{suff_cond}(b), as $\mrm{ran}(f) \leq K$, clearly if $x \notin K$, we have that $f(x) \neq x$. On the other hand, if $x \in K \setminus K_0$, then, recalling $f \restriction \bigoplus_{n > 0} K_n$ is zero, as $f(x) \in K_0$, $f(x) \neq x$. Finally if $x \in K_0$, then we use the choice in $(*_9)$.
Finally, concerning \ref{suff_cond}(c), set $z = y_0 + z_{(p_0, 0)}$, we want to show that for every $x \in G$, $z \neq x - f(x)$. We distinguish cases:
\newline \underline{Case~1}. $x \in K_0$.
\newline In this case we use the choice of $h^0_0$ from $(*_9)$.
\newline \underline{Case~2}. $x \in K \setminus K_0$.
\newline In this case $f(x) \in K_0 \leq K$ so $x - f(x) \in K \setminus K_0$ but $z \in K_0$, so $x - f(x) \neq z$.
\newline \underline{Case~3}. $x \in G \setminus K$.
\newline By $(*_{10})$, $f(x) \in K$, hence $x - f(x) \in G \setminus K$, but $z \in K$, so $x - f(x) \neq z$.
\end{proof}

	To follow there is a strengthening or \ref{no_cohopfian_1} with a more complicated proof.

	\begin{theorem}\label{no_cohopfian_2} Let $\lambda^{\aleph_0} > \lambda > 2^{\aleph_0}$, then:
	\begin{enumerate}[(1)]
	\item no $G \in \mrm{AB}_\lambda$ is co-Hopfian;
	\item if $G \in \mrm{AB}_\lambda$ is reduced, $|G/\mrm{Tor}(G)|^{\aleph_0} = \lambda^{\aleph_0}$ and there is $\mathbb{A} \subseteq \mathbb{P}$ such that:
	\begin{enumerate}[(a)]
	\item if $p \in \mathbb{A}$, then $\exists \, K_p = \bigoplus_{n < \omega} K_{(p, n)} \leq_* G$, $K_{(p, n)} \neq \{0\}$ a finite $p$-group;
	\item if $p \in \mathbb{P} \setminus \mathbb{A}$, then $G/\mrm{Tor}_p(G)$ is $p$-divisible;
	\end{enumerate}
\end{enumerate}		
	then we have that $\lambda^{\aleph_0} \leq |\{ h  \in \mrm{End}(G, K\}|$, where $K = \bigoplus_{p \in \mathbb{A}} K_p$.
\end{theorem}

	\begin{proof} We first prove (2). 
	\begin{enumerate}[$(*_0)$]
	\item W.l.o.g. $K_{(p, 0)}$ is as in the proof of \ref{no_cohopfian_2}, i.e., $K_{(p, 0)} = \mathbb{Z} y_{(p, 0)} \oplus \mathbb{Z} y_{(p, 1)}$, with $y_{(p, \ell)}$ of order $k_{(p, \ell)}$ with $1 \leq k_{(p, 0)} \leq k_{(p, 1)}$.
\end{enumerate}
\begin{enumerate}[$(*_1)$]
	\item Let $\mu = \mrm{min}\{\mu \leq \lambda : \mu^{\aleph_0} \geq \lambda\}$, then:
	\begin{enumerate}
	\item $\mu > 2^{\aleph_0}$;
	\item $\theta < \mu \Rightarrow \theta^{\aleph_0} < \mu$;
	\item $\mu^{\aleph_0} = \lambda^{\aleph_0}$;
	\item $\mrm{cf}(\mu) = \aleph_0$;
	\item W.l.o.g. $\mu = \sum_{n < \omega} \lambda_n$, $2^{\aleph_0} < \lambda_n = \lambda_n^{\aleph_0} < \lambda_{n+1}$;
	\item $|G/\mrm{Tor}(G)| \geq \mu$.
	\end{enumerate}
\end{enumerate}
Why (a)? As $\lambda > 2^{\aleph_0}$. Why (b)? If $\theta < \mu \leq \theta^{\aleph_0}$, then $\lambda \leq \mu^{\aleph_0} \leq (\theta^{\aleph_0})^{\aleph_0} = \theta^{\aleph_0 \aleph_0} = \theta^{\aleph_0}$, a contradiction. Why (c)? As $\lambda^{\aleph_0} \leq (\mu^{\aleph_0})^{\aleph_0} = \mu^{\aleph_0 \aleph_0} = \mu^{\aleph_0} \leq \lambda^{\aleph_0}$. Why (d)? If not then $\mu^{\aleph_0} = |\{ \eta: \eta \in \mu^{\omega} \}| = |\bigcup_{\alpha < \mu}\{ \eta: \eta \in \alpha^{\omega} \}| \leq \sum_{\alpha < \mu} |\alpha|^{\aleph_0} \leq \mu \times \mu = \mu \leq \lambda < \lambda^\aleph_0$. Why (e)? Because of (a)-(d). Why (f)? As $|G/\mrm{Tor}(G)|^{\aleph_0} = \lambda^{\aleph_0}$.

\begin{enumerate}[$(*_{1.5})$]
	\item Let $(x^{*}_\alpha + \mrm{Tor}(G) : \alpha < \lambda_* = |G/\mrm{Tor}(G)|)$ be a basis of $G/\mrm{Tor}(G)$;
	\end{enumerate}
\begin{enumerate}[$(*_2)$]
	\item Let $S_n = \prod_{\ell \leq n} \lambda_\ell$.
\end{enumerate}
\begin{enumerate}[$(*_3)$]
	\item We can find $(G_n, H_n, \bar{x}_n, p_n : n < \omega)$ such that:
	\begin{enumerate}[(a)]
	\item $\mrm{Tor}(G) \leq G_0$ and $G_n \preccurlyeq_{\mathfrak{L}_{\aleph_1, \aleph_0}} G$ (so $G_n \leq_* G$);
	\item $G_n \leq_* G_{n+1}$;
	\item $|G_n| = \lambda_n$;
	\item $G_n \leq_* H_n = G_n \oplus \bigoplus_{\eta \in S_n} L_{(n, \eta)} \leq_* G_{n+1}$, where $L_{(n, \eta)} = \langle x_{(n, \eta)} \rangle^*$;
	\item $p_n \in \mathbb{A}$ and $x_{(n, \eta)}$ is not divisible by $p_n$.
	\end{enumerate}
\end{enumerate}
We prove $(*_3)$. Let $G_0 \preccurlyeq_{\mathfrak{L}_{\aleph_1, \aleph_0}} G$ be of cardinality $\lambda_0$ (cf. \cite[Corollary~3.1.2]{dickman}). Suppose that $G_n$ was chosen, we shall choose $(G_{(n ,\alpha)}, x_{(n, \alpha)} : \alpha < \lambda^+_n)$ as follows: 
	\begin{enumerate}[$(\cdot_1)$]
	\item $G_{(n ,\alpha)} \preccurlyeq_{\mathfrak{L}_{\aleph_1, \aleph_0}} G$;
	\item $x_{(n, \alpha)} \in G \setminus G_{(n ,\alpha)}$ and such that $x_{(n, \alpha)} + G_{(n ,\alpha)} \notin \mrm{Tor}(G/G_{(n ,\alpha)})$;
	\item $G_n \cup \bigcup \{ G_{(n ,\beta)}, x_{(n ,\beta)} : \beta < \alpha\} \subseteq G_{(n ,\alpha)}$.
	\end{enumerate}
As in \ref{no_cohopfian_1}, since $\mrm{Tor}(G) \leq G_0$, w.l.o.g. $G_{(n, \alpha)} \oplus \langle x_{(n, \alpha)} \rangle^*_G \leq_* G$ and let $L_{(n, \alpha)} = \langle x_{(n, \alpha)} \rangle^*_G$.
Let $p_{(n, \alpha)} \in \mathbb{A}$ be such that $L_{(n, \alpha)}$ is not $p_{(n, \alpha)}$-divisible (recalling $G$ is reduced). W.l.o.g. $p_{(n, \alpha)} = p_n$, as $\lambda^+_n$ has uncountable cofinality. Lastly, let $H_n = \bigoplus_{\alpha < \lambda^+} L_{(n, \alpha)} \oplus G_{(n, \alpha)}$.
We can prove by induction on $\alpha \leq \lambda^+_n$ that $G_n \oplus \bigoplus_{\eta \in S_n} L_{(n, \eta)} \leq_*G$, so indeed $H_n \leq_* G_{n+1}$.
As $\lambda_n = |S_n|$ renaming we are done. Choose now $G_{n+1} \preccurlyeq_{\mathfrak{L}_{\aleph_1, \aleph_0}} G$ such that $|G_{n+1}| = \lambda_{n+1}$ and $\bigcup_{\alpha < \lambda^+} G_{(n, \alpha)} \leq G_{n+1}$. 
\begin{enumerate}[$(*_4)$]
	\item For $n < \omega$, let $\mrm{AP}_n$ be the set of $(H, \bar{f})$ such that:
	\begin{enumerate}[(a)]
	\item $H_n \leq H \leq_* G$;
	\item $\bar{f} = \{f_{(n, \eta)} : \eta \in S_n\}$;
	\item $f_{(n, \eta)} \in \mrm{Hom}(H, K_{(p_n, n)})$;
	\item $f_{(n, \eta)}(x_{(n, \nu)}) = 0$ iff $\eta \neq \nu$;
	\item $f_{(n, \eta)} \restriction (G_n)$ is $0$;
	\item if $z \in H$, then $|\{ \eta \in S_n : f_{(n, \eta)}(z) \neq 0 \}| \leq 2^{\aleph_0}$ and even $\leq \aleph_0$;
	\item if $n = 0$, so necessarily $\eta = ()$, then $f_{(0, \eta)}$ is as $(*_{9})(a)$ of the proof of \ref{no_cohopfian_1}.
	\end{enumerate}
\end{enumerate}	
\begin{enumerate}[$(*_{4.5})$]
\item Let $(H, \bar{f}) \leq{\mrm{AP}_n} (H', \bar{f}')$ be the natural order between objects as in $(*_{4})$, that is $H \leq H'$ and $\eta \in S_n$ implies $f_{(n, \eta)} \subseteq f'_{(n, \eta)}$.
\end{enumerate}
\begin{enumerate}[$(*_5)$]
	\item For $n < \omega$, $\mrm{AP}_n \neq \emptyset$.
\end{enumerate}
We prove $(*_5)$. Let $H = H_n$ and let $f_{(n, \eta)} \in \mrm{Hom}(H)$ be such that:
	\begin{enumerate}[(i)]
	\item $f_{(n, \eta)}$ is zero on $G_n$;
	\item $f_{(n, \eta)}$ is zero on $L_{(n, \nu)}$, for $\nu \in S_n \setminus \{\eta\}$;
	\item $\mrm{ran}(f_{(n, \eta)}) \leq K_{(p_n, n)}$;
	\item $f_{(n, \eta)}(x_{(n, \eta)}) \neq 0$.
\end{enumerate}	 
Why we can do this? Cf. $(*_9)$ of the proof of \ref{no_cohopfian_1}.
\begin{enumerate}[$(*_6)$]
	\item If $(H_1, \bar{f}_1) \in \mrm{AP}_n$, then we can find $H_1 \lneq_* H_2 \lneq_* G$ such that $H_2/H_1 \in \mrm{TFAB}$ is of rank $1$ and there is $(H_2, \bar{f}_2) \in \mrm{AP}_n$ such that:
	$$(H_1, \bar{f}_1) <_{\mrm{AP}_n} (H_2, \bar{f}_2) \text{ and } H_1 \nleq H_2.$$
\end{enumerate}
We prove $(*_6)$. Now, for every $\ell < \omega$, we can find $k_\ell < \omega$ and $y_\ell \in H_2$ such that:
\begin{equation}\tag{$\star$} a_\ell := k_\ell y_{\ell+1} - y_\ell \in H_1 \text{ and } H_2 = \bigcup_{\ell < \omega} (\mathbb{Z} y_\ell \oplus H_1).
\end{equation}
Now, for every $\ell < \omega$ and $\eta \in S_n$ we can find $f_{(n, \eta, \ell)} \in \mrm{Hom}(\mathbb{Z} x_\ell \oplus H_1, K_{(p_n, n)})$ extending $f^1_{(n, \eta)}$ such that $f_{(n, \eta, \ell)}(y_\ell) = 0$.
As $K_{(p_n, n)}$ is finite, for some infinite $\mathcal{U}_{(n, \eta)} \subseteq \omega$ we have that, for $\ell_1 < \ell_2 \in \mathcal{U}_{(n, \eta)}$, $f_{(n, \eta, \ell_2)}(y_{\ell_1})$ is constant (why? by the Ramsey Theorem applied on the coloring $f_{(n, \eta, \ell_2)}(y_{\ell_1}) \in K_{(p_n, n)}$). Now, $(f_{(n, \eta, \ell)} : \ell \in \mathcal{U}_{(n, \eta)})$ converges, i.e., if $(k_i : i < \omega)$ lists $\mathcal{U}_{(n, \eta)}$, then $f'_{(n, \eta, k_i)} = f_{(n, \eta, k_{i+1})} \restriction (\mathbb{Z} y_{k_i} \oplus H_i)$ is increasing.
Let $\bar{f}_2 = (f^2_{(n, \eta)} : \eta \in S_n)$, where we let $f^2_{(n, \eta)} = \bigcup \{f'_{(n, \eta, \ell)} : \ell \in \mathcal{U}_{(n, \eta)}\}$. So $(H_2, \bar{f}_2)$ is well-defined and easily it is as required, where the main point is checking $(*_4)(f)$ which is easy as we have:
\begin{enumerate}[$(*_{6.5})$]
	\item if $\eta \in S_n$ and $\bigwedge_{\ell < \omega} f^1_{(n, \eta)}(a_\ell) = 0$, then:
	\begin{enumerate}[(a)]
	\item if $\ell < m$, then $f_{(n, \eta, m)}(y_\ell) = 0$;
	\item $f^2_{(n, \eta)}(y_m) = 0$, for $m < \omega$;
	\item as in (b) for $k y_\ell + b$ ($k \in \mathbb{Z}$ and $b \in H_2$).
	\end{enumerate}
\end{enumerate}
Why? Clauses (b) and (c) are easy and clause (a) can be proved by downward induction on $\ell$, where for $\ell = 1$, the conclusion is true by choice and for $\ell-1$ we use $(\star)$. Hence, we are done proving $(*_6)$.
\begin{enumerate}[$(*_7)$]
	\item For each $n < \omega$ we can choose $\bar{f}_n$ such that $(G, \bar{f}_n) \in \mrm{AP}_n$.
\end{enumerate}
Why? By $(*_5)$ and $(*_6)$ (and their proof). We elaborate. By induction on $\alpha \leq \lambda_*$ we choose pairs $(H^n_\alpha, \bar{f}^n_\alpha)$ such that:
\begin{enumerate}[$(*_{7.5})$]
	\item
	\begin{enumerate}[$(a)$]
	\item $(H^n_\alpha, \bar{f}^n_\alpha) \in \mrm{AP}_n$;
	\item if $\beta < \alpha$, then $(H^n_\beta, \bar{f}^n_\beta) \leq_{\mrm{AP}} (H^n_\alpha, \bar{f}^n_\alpha)$;
	\item if $\alpha = \beta +1$, then $x^*_\beta \in H^n_\alpha$.
	\end{enumerate}
\end{enumerate}
Why we can carry the induction? For $\alpha = 0$, use $(*_5)$. For $\alpha = \beta + 1$, if $x^*_\beta \in H^n_\beta$, let $(H^n_\beta, \bar{f}^n_\beta) = (H^n_\alpha, \bar{f}^n_\alpha)$, while if $x^*_\beta \notin H^n_\beta$, use $(*_6)$. For $\alpha$ limit, let:
$$H^n_\alpha = \bigcup_{\beta < \alpha} H^n_\beta \text{ and } f^n_{(\alpha, \eta)} = \bigcup_{\beta < \alpha} f^n_{(\beta, \eta)}, \text{ for } \eta \in S_n.$$
Having carried the induction, by the definition of $\mrm{AP}_n$ and the choice of $(x^*_\alpha : \alpha < \lambda_*)$ necessarily $H^n_{\lambda_*} = G$ and so we are done proving $(*_7)$.
\begin{enumerate}[$(*_8)$]
	\item If $z \in G$, then $\Lambda_z = \{\nu \in \prod_{n < \omega} \lambda_n : \exists^\infty n (f_{(n, \nu \restriction n}(z) \neq 0 )\}$ has size $\leq 2^{\aleph_0}$.
\end{enumerate}
Why? For each $n < \omega$, $|\{\eta \in \prod_{\ell < n} \lambda_\ell : f_{(n, \eta)}(z) \neq 0\}| \leq \aleph_0$. 
\newline Let $\Lambda = \bigcup \{\Lambda_z : z \in G\}$, so clearly $\Lambda \subseteq \prod_{\ell < \omega} \lambda_\ell$ and $|\Lambda| \leq 2^{\aleph_0} + |G| < \lambda^{\aleph_0} = \prod_{\ell < \omega} \lambda_\ell$ (cf. $(*_1)(c)-(e)$). So for each $\nu \in \prod_{\ell < \omega} \lambda_\ell \setminus \Lambda$ we have:
	\begin{enumerate}[$(*_9)$]
	\item $f_\nu:G \rightarrow K$ defined by $f_\nu(z) = \sum \{f_{\nu \restriction n}(z) : n < \omega\}$
is well-defined.
	\end{enumerate}
Why? As for each $z \in G$ all but finitely many terms in the sum are zero. Hence:
	\begin{enumerate}[$(*_{10})$]
	\item if $\eta \neq \nu \in \prod_{\ell < \omega} \lambda_\ell \setminus \Lambda$, then:
	\begin{enumerate}[(a)]
	\item $f_\nu \in \mrm{Hom}(G, K)$;
	\item then $f_\nu \neq f_\eta$.
	\end{enumerate}
	\end{enumerate}
As necessarily $\prod_{\ell < \omega} \lambda_\ell \setminus \Lambda$ has cardinality $\prod_{\ell < \omega} \lambda_\ell \setminus \Lambda = \lambda^{\aleph_0}$ we are done proving Part (2) of the theorem. Concerning Part (1), note that for each $\nu \in \prod_{\ell < \omega} \lambda_\ell \setminus \Lambda$ we have that $(f_{\nu \restriction \ell} : \ell < \omega)$ is as in the proof of \ref{no_cohopfian_1}, where the only missing part is to justify that $f_\nu$ as in $(*_{9})$ is well-defined, which we do there. Hence,  for each $\nu \in \prod_{\ell < \omega} \lambda_\ell \setminus \Lambda$ we have that $f_\nu$ is as in \ref{suff_cond} and so $G$ is not co-Hopfian.
\end{proof}

	\begin{remark} Similarly to \ref{no_cohopfian_2} we can prove (A) implies (B), where:
	\begin{enumerate}[(A)]
	\item $G \in \mrm{TFAB}_\lambda$ is reduced, $\lambda^{\aleph_0} > \lambda > 2^{\aleph_0}$, $\emptyset \neq \mathbb{A} \subseteq \mathbb{P}$, $p \in \mathbb{A} \Rightarrow G$ is $p$-divisible, and $K = \bigoplus_{p \in \mathbb{A}} K_p$, with $K_p$ as in \ref{no_cohopfian_2}(2);
	\item $|\mrm{Hom}(G, K)| \geq \lambda^{\aleph_0}$.
	\end{enumerate}
\end{remark}

	\begin{ntheorem2} If $2^{\aleph_0} < \lambda < \lambda^{\aleph_0}$, and $G \in \mrm{AB}_\lambda$, then $G$ is not co-Hopfian.
\end{ntheorem2}

\begin{proof} Immediate by Theorem~\ref{no_cohopfian_2}.
	\end{proof}

\section{Absolutely Hopfian abelian groups}\label{last_sec}

	In reading Convention~\ref{conv_pp_frm} and subsequent items recall Notation~\ref{inf_logic_notation}.

	\begin{convention}\label{conv_pp_frm} By a positive conjunctive existential $\varphi(\bar{x}_n) \in \mathfrak{L}_{\infty, \aleph_0}(\tau_{\mrm{AB}})$ we mean a formula of $\mathfrak{L}_{\infty, \aleph_0}(\tau_{\mrm{AB}})$ which does not uses $\neg$, $\vee$ and $\forall$.
\end{convention}

\begin{fact}\label{fact_def_subgroup} Let $\varphi(\bar{x}_n) \in \mathfrak{L}_{\infty, \aleph_0}(\tau_{\mrm{AB}})$ be positive conjunctive existential and $G \in \mrm{AB}$. 
	\begin{enumerate}[(A)]
	\item $\varphi(G) = \{\bar{a} \in G^n : G \models \varphi[\bar{a}]\}$ is a subgroup of $G$;
	\item if $f \in \mrm{End}(G)$ and $G \models \varphi[\bar{a}]$, then $G \models \varphi[f(\bar{a})]$.
	\end{enumerate}
\end{fact}

	\begin{proof} Clause (A) is by e.g. \cite[Claim~2.3]{977}. Clause (B) is easy.
	\end{proof}

	\begin{fact}\begin{enumerate}[(1)]
	\item If $\lambda$ is beautiful $>|R|$ and $M$ is an $R$-module of cardinality $\geq \lambda$, then $M$ is not absolutely co-Hopfian.
	\item If $\lambda$ is $<$ the first beautiful cardinal, then there is $G \in \mathrm{TFAB}$ of cardinality $\lambda$ which is absolutely endo-rigid (and thus Hopfian).	
\end{enumerate}
\end{fact}

	\begin{proof} (1) is by the proof of \cite[Theorem~4]{678}. (2) is by \cite{abs_ind}.
\end{proof}

	The use of the forcing $\mrm{Levy}(\aleph_0, |G|)$ in the proof of Theorem~\ref{first_main_th} is justified by:

	\begin{fact} For given $G \in \mrm{TFAB}_\lambda$, the following are equivalent:
	\begin{enumerate}[(a)]
	\item $\mrm{Levy}(\aleph_0, \lambda)$ forces ``$G$ is not Hopfian'';
	\item some forcing $\mathbb{P}$ forces ``$G$ is not Hopfian'';
	\item every forcing $\mathbb{P}$ collapsing $\lambda$ to $\aleph_0$ forces ``$G$ is not Hopfian''.
	\end{enumerate}
\end{fact}

	\begin{convention}\label{convention_Levy} In the proof below by ``absolutely if $f \in \mrm{End}(G)$, then...'' we mean that the forcing $\mrm{Levy}(\aleph_0, |G|)$ forces the statement ``if $f \in \mrm{End}(G)$, then ...''.
\end{convention}

\begin{ntheorem3} For all $\lambda \in \mrm{Card}$ there is $G \in \mrm{TFAB}_\lambda$ which is absolutely Hopfian.
\end{ntheorem3}

	\begin{proof} Let $\lambda$ be an infinite cardinal. We want to construct $G \in \mrm{TFAB}_\lambda$ which is absolutely Hopfian. To this extent, let:
	\begin{enumerate}[(a)]
	\item for $n < \omega$, $\mrm{decr}_n(\lambda) = \{\eta : \eta \text{ is a decreasing $n$-sequence of ordinals} < \lambda \}$;
	\item $\mrm{decr}_{\geq 2}(\lambda) = \bigcup_{2 \leq n < \omega} \mrm{decr}_n(\lambda)$;
	\item $\mrm{decr}(\lambda) = \bigcup_{n < \omega} \mrm{decr}_n(\lambda)$.
	\end{enumerate}
	\begin{enumerate}[$(*_1)$]
	\item let $p_1, p_2$, $p_{(1, n)}$ ($n \geq 1$), $p_{(2, n)}$ ($n \geq 1$), $q_{(1, n)}$ ($n \geq 1$), $q_{(2, n)}$ ($n \geq 1$) be pairwise distinct primes (notice that we can replace $\mathbb{Z}$ by a ring $R$ with such primes, certainly if $R$ is an integral domain);
\end{enumerate}
\begin{enumerate}[$(*_2)$]
	\item Let $(\eta_\alpha : \alpha < \lambda)$ list $\mrm{decr}_{\geq 2}(\lambda)$ with no repetitions.
\end{enumerate}
	\begin{enumerate}[$(*_3)$]
	\item Now, we define:
	\begin{enumerate}[$(\cdot_1)$]
	\item $H_2 = \{\mathbb{Q}x_{(\ell, \alpha)} : \alpha < \lambda, \ell \in \{1, 2\}\}$;
\item $H_0 = \{\mathbb{Z}x_{(\ell, \alpha)} : \alpha < \lambda, \ell \in \{1, 2\}\}$.
	\end{enumerate}
\end{enumerate}
	\begin{enumerate}[$(*_4)$]
	\item For $\eta \in \mrm{decr}(\lambda) \setminus \{()\}$ and $\ell \in \{1, 2\}$ let $x_{(\ell, \eta)}$ be:
	\begin{enumerate}[$(\cdot_1)$]
	\item $x_{(\ell, (\alpha))} = x_{(\ell, \alpha)}$, for $\alpha < \lambda$;
	\item $x_{(\ell, \eta)} = x_{(3-\ell, \beta)}$, when $\beta < \lambda$ and $\eta = \eta_\beta$ (recall $(*_2)$).
	\end{enumerate}
\end{enumerate}
	\begin{enumerate}[$(*_5)$]
	\item Let $G = H_1 \leq H_2$ be generated by $X_1 \cup X_2 \cup X_3 \cup X_4 \cup X_5 \cup X_6$, where:
	$$X_1 = \{p^{-m}_1 x_{(1, \alpha)} : \alpha < \lambda, m < \omega\};$$
	$$X_2 = \{p^{-m}_2 x_{(2, \alpha)} : \alpha < \lambda, m < \omega\};$$
	$$X_3 = \{p^{-m}_{(1, n)} (x_{(1, \eta)} - x_{(1, \nu)}) : n \geq 1, \eta \in \mrm{decr}_n(\lambda), \nu \in \mrm{decr}_{n+1}(\lambda), \eta \triangleleft \nu, m < \omega\};$$
	$$X_4 = \{p^{-m}_{(2, n)} (x_{(2, \eta)} - x_{(2, \nu)}) : n \geq 1, \eta \in \mrm{decr}_n(\lambda), \nu \in \mrm{decr}_{n+1}(\lambda), \eta \triangleleft \nu, m < \omega \};$$
	$$X_5 = \{q^{-m}_{(1, n)} x_{(1, \eta)} : \eta \in \mrm{decr}_{\geq n}(\lambda), m < \omega, 2 \leq n < \omega \};$$
	$$X_6 = \{q^{-m}_{(2, n)} x_{(2, \eta)} : \eta \in \mrm{decr}_{\geq n}(\lambda), m < \omega, 2 \leq n < \omega \}.$$
\end{enumerate}
\begin{enumerate}[$(*_6)$]
	\item For $\ell \in \{1, 2 \}$ and $\alpha < \lambda$, let:
	\begin{enumerate}[$(a)$]
	\item $G^0_\ell = \langle X_\ell \rangle_G$, $G_\ell = \langle X_\ell \rangle^*_G$, where $X_\ell = \{x_{(\ell, \beta)} : \beta < \lambda\}$;
	\item $G^0_{(\ell, \alpha)} = G^0_{(\ell, \alpha, 1)} = \langle \{ x_{(\ell, \beta)} : \alpha \leq \beta < \lambda\}\rangle_G$ and $G_{(\ell, \alpha)} = G_{(\ell, \alpha, 1)} = \langle \{ x_{(\ell, \beta)} : \alpha \leq \beta < \lambda\}\rangle^*_G = \langle \{ x_{(\ell, \beta)} : \mrm{lg}(\eta_\beta) = 1 \text{ and } \alpha \leq \mrm{min}(\mrm{ran}(\eta_\beta))\} \rangle^*_G$;
	\item for $n \geq 2$, $G_{(\ell, \alpha, n)} = \langle \{ x_{(3-\ell, \beta)} : \mrm{lg}(\eta_\beta) = n \text{ and } \alpha \leq \mrm{min}(\mrm{ran}(\eta_\beta))\} \rangle^*_G$.
	\end{enumerate}
\end{enumerate}
Notice that:
\begin{enumerate}[$(*_{7})$]	
	\item
\begin{enumerate}[(a)]
	\item $G = \langle G_1 \oplus G_2 \rangle^*_G$;
	\item $G_{(\ell, \alpha, n)} \leq_* G$;
	\item for $\ell \in \{1, 2\}$ and $n \geq 1$, the sequence $(G_{(\ell, \alpha, n)} : \alpha < \lambda)$ is $\subseteq$-decreasing, continuous and with intersection $\{ 0 \}$.
\end{enumerate}
\end{enumerate}
\begin{enumerate}[$(*_8)$]
	\item For $\ell \in \{1, 2 \}$ let $\psi_\ell(x) = \bigwedge_{n < \omega} p^n_\ell  \, \vert \,x$.
\end{enumerate}
\begin{enumerate}[$(*_9)$]
	\item For $\ell \in \{1, 2 \}$ we have:
	\begin{enumerate}[$(a)$]
	\item $\psi_\ell(x)$ is a formula in $\mathfrak{L}_{\aleph_1, \aleph_0}(\tau_{\mrm{AB}})$;
	\item $\psi_\ell(x)$ is positive conjunctive existential;
	\item $\psi_\ell(G) = G_\ell$;
	\item if $f \in \mrm{End}(G_\ell)$, then $f$ maps $G_\ell$ into $G_\ell$.
	\end{enumerate}
\end{enumerate}
	[Why? Clauses (a), (b) are clear, clause (d) follows by clause (c) and Fact~\ref{fact_def_subgroup}(B), and clause (c) is clear by Fact~\ref{fact_def_subgroup}(A) and the definitions, recalling that if $L \in \mrm{TFAB}$, and $p$ is a prime, then the $p^\infty$-divisible elements of $L$ form a pure subgroup of~$L$.]
\begin{enumerate}[$(*_{10})$]
	\item For $1 \leq n < \omega$ and $\ell \in \{1, 2\}$, by induction on $\alpha < \lambda$ we define $\varphi_{(\ell, \alpha, n)}(x)$ as:
	\begin{enumerate}[$(a)$]
	\item if $\alpha = 0$ and $n =1$, then $\varphi_{(\ell, \alpha, n)}(x) = \psi_\ell(x)$;
	\item if $\alpha = 0$ and $n > 1$, then $\varphi_{(\ell, \alpha, n)}(x) = \psi_{3-\ell}(x) \wedge \bigwedge_{m \geq 1} \exists y(q^m_{(\ell, n)} y = x)$;
	\item if $\alpha > 0$, then $\varphi_{(\ell, \alpha, n)}(x)$ is the formula:
	$$\bigwedge_{\beta < \alpha} \exists y (\varphi_{(\ell, \beta, n+1)}(y) \wedge p^\infty_{(\ell, n)} \vert \, (x-y) \wedge \varphi_{(\ell, \beta, n)}(x) \wedge q^\infty_{(\ell, n+1)} \, \vert \,y);$$
	\item $\varphi^*_{(\ell, 0, n)}(x) =\varphi_{(\ell, 0, n)}(x)$ and $\alpha > 0 \Rightarrow \varphi^*_{(\ell, \alpha, n)}(x) = \bigvee_{m \geq 1} \varphi_{(\ell, \alpha, n)}(mx)$.
	\end{enumerate}
\end{enumerate}
\begin{enumerate}[$(*_{10.5})$]
	\item $\varphi^*_{(\ell, \alpha, n)}(G)$ is the pure closure of $\varphi_{(\ell, \alpha, n)}(G)$ (we shall use this freely).
\end{enumerate}
\begin{enumerate}[$(*_{11})$]
	\item
	\begin{enumerate}[$(a)$]
	\item $\varphi_{(\ell, \alpha, n)}(x) \in \mathfrak{L}_{\lambda, \aleph_0}(\tau_{AB})$ is positive conjunctive existential;
	\item $\varphi^*_{(\ell, 0, 1)}(G) = G_{(\ell, 0, 1)} = G_{(\ell, 0)} = G_\ell$;
	\item $\varphi^*_{(\ell, \alpha, 1)}(G) = G_{(\ell, \alpha, 1)} = G_{(\ell, \alpha)}$;
	\item $\varphi^*_{(\ell, \alpha, n)}(G) = G_{(\ell, \alpha, n)}$.
\end{enumerate}
\end{enumerate}
We prove $(*_{11})$ by induction on $\alpha < \lambda$.
\newline \underline{Case 1}. $\alpha = 0$. Easy.
\newline \underline{Case 2}. $\alpha$ limit. Easy.
\newline \underline{Case 3}. $\alpha = \beta+1$.
\newline \underline{Case (a) of $(*_{11})$}. Just read the definition of $\varphi_{(\ell, \alpha, 1)}$.
\newline \underline{Case (b) of $(*_{11})$}. Just read the definition of $\varphi_{(\ell, \alpha, 1)}$ and $\varphi^*_{(\ell, \alpha, 1)}$.
\newline \underline{Case (c), (d) of $(*_{11})$}. The proofs of (c) and (d) are similar, so we write only the proof of (c). Note that proving (c) we use clauses (c) and (d) for all $\beta < \alpha$. 
\newline Thus, we want to prove:
\begin{enumerate}[(i)]
	\item if $\gamma \in [\alpha, \lambda)$, then $x_{(\ell, \gamma)} \in \varphi_{(\ell, \alpha, 1)}(G)$;
	\item if $x \in H_0$ (cf. $(*_3)$) and $x \in \varphi_{(\ell, \alpha, 1)}(G)$, then $x \in G_{(\ell, \alpha, 1)}$.
\end{enumerate}
We prove (i). We have to show that letting $x = x_{(\ell, \gamma)}$ for every $\beta_1 \leq \beta$ we have:
\begin{equation} \tag{$\star_1$} G \models \exists y (\varphi_{(\ell, \beta_1, 2)}(y) \wedge p^\infty_{(\ell, 1)} \vert \, (x-y) \wedge \varphi_{(\ell, \beta_1, 1)}(x) \wedge q^\infty_{(\ell, 2)} \, \vert \,y).
\end{equation}
Hence, we have to find a witness for $(\star_1)$, to this extent we let $y = x_{(\ell, (\gamma, \beta_1))}$ (cf. $(*_4)(\cdot_2)$) and show that this choice of $y$ is as wanted. Now, the first conjunct $\varphi_{(\ell, \beta_1, 2)}(y)$ holds by the inductive hypothesis noticing that $y = x_{(\ell, (\alpha, \beta_1))} \in G_{(\ell, \beta_1, 2)}$ (and recalling that we are doing an induction on $\alpha$ for all $1 \leq n < \omega$ for clauses (c) and (d) simultaneously).
The second conjunct $p^\infty_{(\ell, 1)} \vert \, (x-y)$ holds by the choice of $G$ (cf. $X_3$ and $X_4$ of $(*_5)$).
The third conjunct $\varphi_{(\ell, \beta_1, 1)}(x)$ holds by the inductive hypothesis (as $x = x_{(\ell, \gamma)} \in G_\ell$).
Finally, the fourth conjunct $q^\infty_{(\ell, 2)} \, \vert \,y$ holds by the choice of $G$ (cf. $X_5$ and $X_6)$ of $(*_5)$. This concludes the proof of (i).

\smallskip

\noindent We now prove (ii). So let $x \in H_0$ and $x \in \varphi_{(\ell, \alpha, 1)}(G)$, we want to show that $x \in G_{(\ell, \alpha, 1)}$. Clearly $x \in \varphi_{(\ell, \alpha, 1)}(G)$ implies that $x \in G_{\ell}$, in fact as $x \in \varphi_{(\ell, \alpha, 1)}(G)$ in particular $G \models \varphi_{(\ell, \beta, 1)}(x)$ (as this is the third conjunct of $\varphi_{(\ell, \alpha, 1)}$, see $(\star_1)$ above with $\beta_1 = \beta$), so by the inductive hypothesis we have that $x \in G_\ell$ and in fact as $x \in H_0$ we have that $x \in G^0_\ell$ (cf. $(*_6)(b)$). Now, toward contradiction assume $x \notin G_{(\ell, \alpha, 1)}$, so $x \neq 0$. As $x \in G^0_\ell = \langle \{x_{(\ell, \gamma)} : \gamma < \lambda \rangle_G$ and $x \neq 0$ there are $k < \omega$ and $\alpha_0 < \cdots < \alpha_k < \lambda$ such that we have the following equation:
\begin{equation} \tag{$\star_2$} x = \sum_{i \leq k} n_i x_{(\ell, \alpha_i)},
\end{equation}
with $n_i \in \mathbb{Z} \setminus \{0\}$. Now, if $\alpha_0 \geq \alpha$ we get the desired conclusion, so we assume that $\alpha_0 < \alpha$. Now, if $\alpha_0 < \beta$, clearly $x \notin G_{(\ell, \beta, 1)}$, as $\{x_{(\ell, \gamma)} : \gamma \in [\beta, \lambda)\}$ is a basis of $G_{(\ell, \beta, 1)}$, but this contradicts the inductive hypothesis. Hence,  w.l.o.g. we can assume that $\alpha_0 = \beta$. Now, as $x \in \varphi_{(\ell, \alpha, 1)}(G)$ and $\beta < \alpha$ there is $y_0 \in G$ such that:
\begin{equation} \tag{$\star_3$} G \models \varphi_{(\ell, \beta, 2)}(y_0) \wedge p^\infty_{(\ell, 1)} \vert \, (x-y_0) \wedge \varphi_{(\ell, \beta, 1)}(x) \wedge q^\infty_{(\ell, 2)} \, \vert \,y_0.
\end{equation}
Also, for some $m < \omega$ we have that $y = my_0 \in H_0$ and easily we have:
\begin{equation} \tag{$\star_4$} G \models \varphi_{(\ell, \beta, 2)}(y) \wedge p^\infty_{(\ell, 1)} \vert \, (mx-y) \wedge \varphi_{(\ell, \beta, 1)}(mx) \wedge q^\infty_{(\ell, 2)} \, \vert \,y.
\end{equation}
Now, by the fact that $G \models q^\infty_{(\ell, 2)} \, \vert \,y$ and $y  \in H_0$ there are pairwise distinct $\eta_0, ..., \eta_{i-1} \in \mrm{decr}_2(\lambda)$ and $m_j \in \mathbb{Q}$ such that we have the following:
\begin{equation} \tag{$\star_5$}y = \sum_{j < i} m_j x_{(\ell, \eta_j)}.
\end{equation}
By $(\star_4)$, $G \models p^\infty_{(\ell, 1)} \vert (mx - y)$. Now, $\{z \in G : p^\infty_{(\ell, 1)} \vert z\}$ is a pure subgroup of $G$ and its intersection with $H_0$ is generated by (recalling that $x_{(\ell, (\xi))} = x_{(\ell, \xi)}$, cf. $(*_4)(\cdot_1)$):
\begin{equation} \tag{$\star_6$}  \{x_{(\ell, (\zeta))} - x_{(\ell, (\zeta, \epsilon))} : \epsilon < \zeta < \lambda\}.
\end{equation}
Why $(\star_6)$? By $X_3$ and $X_4$ in $(*_5)$. So for some $\epsilon_j < \zeta_j < \lambda$, with $j < j_*$, we have:
\begin{equation} \tag{$\star_7$} 
	mx-y = \sum_{j < j_*} n'_j (x_{(\ell, \zeta_j)} - x_{(\ell, (\zeta_j, \epsilon_j)}),
\end{equation}
for $n'_j \in \mathbb{Z} \setminus \{0\}$. Also, by $(\star_2)$ and $(\star_5)$ we have that:
\begin{equation} \tag{$\star_8$} 
 mx-y = m\sum_{i \leq k} n_i x_{(\ell, \alpha_i)} - \sum_{j < i} m_j x_{(\ell, \eta_j)}.
\end{equation}
Recall also (crucially) that we are under the following assumption:
\begin{equation} \tag{$\star_9$} 
\alpha_0 = \beta.
\end{equation}
W.l.o.g. $(\zeta_j : j < j_*)$ is non decreasing and $j_1 < j_2 \wedge \zeta_1 = \zeta_2$ implies $\epsilon_{j_1} < \epsilon_{j_2}$. We now compare the supports in $(\star_7)$ and $(\star_8)$. There are three cases:
\newline \underline{Case A}. $\zeta_0 < \beta$.
\newline In this case $x_{(\ell, \zeta_0)}$ appears in $(\star_7)$ but not in $(\star_8)$ (recall $\alpha_0 = \beta$), a contradiction.
\newline \underline{Case B}. $\zeta_0 > \beta$.
\newline In this case  $x_{(\ell, \beta)}$ appears in $(\star_8)$ but not in $(\star_7)$ (recall $\alpha_0 = \beta$), a contradiction.
\newline \underline{Case C}. $\zeta_0 = \beta = \alpha_0$.
\newline In this case we compare for $\nu \in \mrm{decr}_{2}(\lambda)$ when $x_{(\ell, \nu)}$ is in the support of $(\star_7)$ and when it is in the support of $(\star_8)$. We restrict ourselves to the case $\nu = (\zeta_0, \epsilon_0) = (\beta, \epsilon_0)$. As $x_{(\zeta_0, \epsilon_0)}$ appears in $(\star_7)$ it has to appear also in $(\star_8)$, so for some $j < i$ we have that $\eta_j = (\beta, \epsilon_0)$, so $x_{(\ell, \eta_j)}$ appears in the support of $y$, but, by $(\star_4)$, $G \models \varphi_{(\ell, \beta, 2)}(y)$ and so we get a contradiction to clause (d) for $\beta$, as $\epsilon_0 < \beta = \zeta_0$ (recalling that $\nu \in \mrm{decr}_2(\lambda)$). This concludes the proof of $(*_{11})$.

\smallskip

\noindent From here on we may work in $\mathbf{V}^{\mrm{Levy}(\aleph_0, \lambda)}$, toward proving that $G$ is absolutely Hopfian, alternatively all the claims below about $f \in \mrm{End}(G)$ can be considered as absolute statements in the sense of Convention~\ref{convention_Levy}.
\begin{enumerate}[$(*_{12})$]
	\item if $f \in \mrm{End}(G)$, $\ell \in \{1, 2\}$ and $\alpha < \lambda$, then:
	\begin{enumerate}[$(\alpha)$]
	\item $f$ maps $G_{(\ell, \alpha)}$ into $G_{(\ell, \alpha)}$;
	\end{enumerate}	
	\begin{enumerate}[$(\beta)$]
	\item $f$ maps $G_{(\ell, \alpha, n)}$ into $G_{(\ell, \alpha, n)}$;
	\end{enumerate}
\end{enumerate}
[Why? By Fact~\ref{fact_def_subgroup}(B), $(*_{11})$ and: $\varphi^*_{(\ell, \alpha, n)}(G)$ is the pure closure of $\varphi_{(\ell, \alpha, n)}(G)$.]
\begin{enumerate}[$(*_{12.5})$]
\item if $f \in \mrm{End}(G)$, then there is a unique $\hat{f} \in \mrm{End}(H_2)$ extending $f$. 
\end{enumerate}
Why? This is because $H_2$ is the divisible hull of $G$ (so $H_2/G$ is torsion) and by the following fact: if $L_1 \leq L_2 \in \mrm{TFAB}$, $L_2/L_1$ is torsion, $L_2$ is divisible and $f \in \mrm{End}(L_1)$, then $f$ has exactly one extension to a map in $\mrm{End}(L_2)$.
\begin{enumerate}[$(*_{13})$]
	\item if $f \in \mrm{End}(G)$ is onto, then $f$ maps $G_{\ell}$ onto $G_{\ell}$.
\end{enumerate}
[Why? Let $\ell \in \{1, 2\}$ and $x \in G_\ell$, so for some $y \in G$ we have $f(y) = x$. As $y \in G$, for some $q_1, q_2 \in \mathbb{Q}$ and $y_1 \in G_1$, $y_2 \in G_2$ we have that $y = q_1y_1 + q_2y_2$ (recall that $G = \langle G_1 + G_2 \rangle^*_G \leq H_2$ so that $q_\ell y_\ell$ make sense). Thus, $f(y) = \hat{f}(y) = \hat{f}(q_1y_1 + q_2y_2) = q_1\hat{f}(y_1) + q_2\hat{f}(y_2) = q_1\hat{f}(y_1) + q_2\hat{f}(y_2) = q_1f(y_1) + q_2f(y_2)$, but, as $x = f(y) \in G_\ell$ and $G_\ell \cap G_{3 -\ell} = \{0\}$, necessarily $q_{3-\ell} = 0$ so that $y = q_\ell y_\ell \in G_\ell$.]
\begin{enumerate}[$(*_{14})$]
	\item if $f \in \mrm{End}(G)$ is onto and  $\alpha < \lambda$, then $f$ maps $G_{(\ell, \alpha)} \setminus G_{(\ell, \alpha+1)}$ into itself.
\end{enumerate}
[Why? We prove this by induction on $\alpha < \lambda$. By the inductive hypothesis and $(*_7)(c)$ $f$ maps $G_{\ell} \setminus G_{(\ell, \alpha)}$ into itself, so by $(*_{13})$ we have that:
\begin{enumerate}[$(*_{14.5})$]
	\item $f$ maps $G_{(\ell, \alpha)}$ onto $G_{(\ell, \alpha)}$.
\end{enumerate}
By $(*_{14.5})$, $x_{(\ell, \alpha)} \in \mrm{ran}(f \restriction G_{(\ell, \alpha)})$, let then $z \in G_{(\ell, \alpha)}$ be such that $f(z) = x_{(\ell, \alpha)}$. Now, $x_{(\ell, \alpha)} \notin G_{(\ell, \alpha+1)}$ by $(*_{6})(b)$ and $(*_{3})$, so $f(z) = x_{(\ell, \alpha)} \notin G_{(\ell, \alpha+1)}$. 
As $z \in G_{(\ell, \alpha)}$ and $G_{(\ell, \alpha)} = \langle G_{(\ell,\alpha+1)} \cup \{ x_{(\ell, \alpha)} \}\rangle^*_G$,
necessarily for some rational $q \neq 0$ and $b \in G_{(\ell, \alpha+1)}$ we have that $z = qx_{(\ell, \alpha)} + b$.
 This implies the following: 
$$\begin{array}{rcl}
y \in G_{(\ell, \alpha)} \setminus G_{(\ell, \alpha+1)} & \Rightarrow  & 
		\exists q_y \in \mathbb{Q}^+, y \in q_y x_{(\ell, \alpha)} + G_{(\ell, \alpha+1)} \\
  & \Rightarrow & f(y) \in q q_y x_{(\ell, \alpha)} + G_{(\ell, \alpha+1)} \\
  & \Rightarrow & f(y) \in G_{(\ell, \alpha)} \setminus G_{(\ell, \alpha+1)}.
\end{array}$$
So $f$ maps $G_{(\ell, \alpha)} \setminus G_{(\ell, \alpha+1)}$ into $G_{(\ell, \alpha)} \setminus G_{(\ell, \alpha+1)}$, as wanted in $(*_{14})$. 
\begin{enumerate}[$(*_{15})$]
	\item if $f \in \mrm{End}(G)$ is onto and  $\alpha < \lambda$, then $f$ maps $G_{(\ell, \alpha)} \setminus G_{(\ell, \alpha+1)}$ onto itself.
\end{enumerate} 
Why? As $f$ maps $G_{(\ell, \alpha)}$ onto $G_{(\ell, \alpha)}$ (by $(*_{14.5})$) and $G_{(\ell, \alpha+1)}$ into $G_{(\ell, \alpha)}$ (by $(*_{15})$).
\begin{enumerate}[$(*_{16})$]
	\item If $f \in \mrm{End}(G)$ is onto, then $f$ is $1$-to-$1$.
\end{enumerate}
[Why? By $(*_{13})$ and $G = \langle G_1 \oplus G_2 \rangle^*_G$, it suffices to show that, fixed $\ell \in \{1, 2\}$, $0 \neq x \in G_\ell$ implies that $f(x) \neq 0$. Let $\alpha < \lambda$ be minimal such that $x \in G_{(\ell, \alpha)} \setminus G_{(\ell, \alpha+1)}$, which is justified as $\bigcap_{\alpha < \lambda} G_{\ell, \alpha} = \{0\}$ and $(G_{(\ell, \alpha)} : \alpha \leq \lambda)$ is $\subseteq$-decreasing continuous, then by $(*_{14})$ we are done, as $f(x) \in G_{(1, \alpha)} \setminus G_{(1, \alpha+1)}$, and so $f(x) \neq 0$.]

\smallskip 
\noindent
As, from $(*_{12})$ on we have been assuming to work in $\mathbf{V}^{\mrm{Levy}(\aleph_0, \lambda)}$, we conclude that $G$ is indeed not only Hopfian but absolutely Hopfian, and so we are done.
\end{proof}

	\begin{remark}\label{remark_4prime} Using ideas on the line of \cite{4_primes} (cf. the use of four primes), the proof of Theorem~\ref{first_main_th} can be simplified using only a small (finite) number of primes, and thus in particular it works for $R$-modules with $R$ having at least that amount of primes, but we choose not to follow this route in order to simplify the proof.
\end{remark}

\end{document}